\documentclass[12pt]{amsart}

\newtheorem{theorem}{Theorem}[section]

\newtheorem{corollary}[theorem]{Corollary}

\newtheorem{lemma}[theorem]{Lemma}
\newtheorem{proposition}[theorem]{Proposition}

\theoremstyle{definition}

\newtheorem{definition}[theorem]{Definition}

\newtheorem{remark}[theorem]{Remark}

\newtheorem{example}[theorem]{Example}

\newtheorem{question}[theorem]{Question}

\theoremstyle{parrafo}

\newcommand{\R}{\mathbb{R}}

\begin{document}

\title[]{Kissing numbers and the centered maximal operator}

\author{J. M. Aldaz}
\address{Instituto de Ciencias Matem\'aticas (CSIC-UAM-UC3M-UCM) and Departamento de 
Mate\-m\'aticas,
Universidad  Aut\'onoma de Madrid, Cantoblanco 28049, Madrid, Spain.}
\email{jesus.munarriz@uam.es}
\email{jesus.munarriz@icmat.es}


\thanks{The author was partially supported by Grant MTM2015-65792-P of the
MINECO of Spain, and also by by ICMAT Severo Ochoa project SEV-2015-0554 (MINECO)}







\begin{abstract} We prove that in a  metric measure space $X$, if for some 
	$p \in (1,\infty)$ there are uniform bounds (independent of the measure) for 
the weak type $(p,p)$  of the centered maximal operator, then $X$ satisfies
a certain geometric condition, 
the  Besicovitch  intersection property, which in turn
implies the uniform  weak type $(1,1)$  of the centered operator.  Thus, the following characterization is obtained:  the centered maximal operator satisfies  uniform weak type $(1,1)$ bounds  if and only if the space
$X$ has the  Besicovitch intersection property.

In $\mathbb{R}^d$ with any norm, the constants coming from the Besicovitch intersection property
are bounded above  by the translative  kissing numbers. The extensive literature on kissing numbers allows us to obtain, first, sharp estimates on the uniform bounds satisfied by the centered maximal operators  defined by arbitrary  norms on the plane, second, sharp estimates in every dimension when the $\ell_\infty$ norm is used, and third, improved 
estimates in all dimensions when considering euclidean balls, 
as well as the sharp
constant in dimension  3. 

Additionally, we prove that the existence of uniform $L^1$ bounds for the averaging operators associated to arbitrary measures and radii, is equivalent to a weaker variant of the Besicovitch intersection property.
\end{abstract}


\maketitle


\markboth{J. M. Aldaz}{The centered maximal operator}

\section {Introduction} 

It is well known that the centered maximal operator $M_\mu$ 
is of weak type (1,1) for arbitrary, locally finite Borel measures $\mu$ on $\mathbb{R}^d$, with bounds exponential 
in $d$ but  independent of 
the measure,
because
of the Besicovitch covering theorem. 

Here we show, in the context of metric measure
spaces $(X, d, \mu)$, that the full force of the theorem is not needed;
in fact, the exact condition is given by the Besicovitch intersection 
property, which controlls the maximal overlap of families of balls such that each ball does not contain the center of any other ball in the collection. Since in metric spaces, in general neither centers nor radii of balls are unique, the preceding statement needs to be made more precise.

\begin{definition}  \label{BIP} A  collection $\mathcal{C}$ of balls in a metric space $(X, d)$ is a {\em Besicovitch family} if there is a choice function assigning a center and a radius to each ball in $\mathcal{C}$, in such a way that 
	for every pair  of distinct balls $B(x, r), B(y,s) \in \mathcal{C}$,
	$x \notin B (y,s)$ and  $y \notin B (x,r)$. Denote by
	$\mathcal{BF}(X,d)$ the collection of all Besicovitch families
	of $(X, d)$.  The 
	 {\em Besicovitch constant}  of $(X, d)$ is
	 	\begin{equation}\label{BC}
	L(X, d) := \sup \left\{	\sum_{B(x,  r) \in \mathcal{C}} \mathbf{1}_{B(x,  r)} (y): y \in X, \ \mathcal{C} \in \mathcal{BF}(X,d) \right\}.
	 \end{equation}
	 We say that $(X, d)$ has the {\em   Besicovitch
		Intersection Property} if $L(X, d) < \infty$.
\end{definition} 

 One of our main results is the following 

\begin{theorem} \label{main} The Besicovitch constant	$L(X, d)$ is equal to $\sup_{\mu}\|M_\mu \|_{L^1\to L^{1, \infty}}$, 
	where the supremum is taken over all $\tau$-additive, locally finite Borel measures $\mu$ on $(X, d)$. 
\end{theorem}

Thus, $M_\mu$ satisfies  weak type $(1,1)$ bounds  that are uniform in
$\mu$, if and only if 
$X$ has the  Besicovitch intersection property, and the 
optimal constant is the same in both cases. 

We mention that the preceding result applies to all Borel measures in separable metric spaces, since there $\tau$-additivity is automatic (cf. Definition \ref{tau}).

Furthermore, 
if for some $1 < p < \infty$ the centered maximal operator $M_\mu$
satisfies uniform weak type $(p,p)$ bounds,
then  $X$ has the  Besicovitch intersection property. So we obtain
an extrapolation result, from uniform
weak type $(p,p)$ to uniform
weak type $(1,1)$ (cf. Theorem \ref{extrapolation}).

Recall that spaces satisfying the conclusion of the Besicovitch covering theorem tend to be rather special, cf. \cite[pp. 7-8]{He}. The  Besicovitch
intersection property has clear advantages over stronger hypotheses
of 
Besicovitch type:   more spaces have it (for one example, see \cite[Example
3.4]{LeRi}) and
it is easier to 
handle technically (cf. \cite{LeRi}). For us, the main application here is that it leads to substantially better bounds than the previously known  ones, when $(X, d) = (\mathbb{R}^d, \| \cdot\|)$ and $ \| \cdot\|$ is a norm.  The reason why one can use the Besicovitch intersection property instead of stronger hypotheses in order to obtain uniform bounds for the maximal operators, is that arbitrary measures can be replaced by finite sums of weighted Dirac deltas, by the ``local discretization of measures'' presented in Theorem \ref{discretization}.

Considerable
efforts have been made to determine the boundedness properties of $M_\mu$ 
in many  classes of spaces (as a very small
 sample, we mention  \cite{Io}, \cite{Li},  \cite{Li1}, \cite{NaTa},
\cite{Str}).  When
boundedness is known, it is often interesting   to improve on the constants, finding the sharp
ones if possible.  Starting with the work of E. M. Stein, cf. \cite{St},
and J. Bourgain, cf. \cite{Bou}, the case of Lebesgue measure in $\mathbb{R}^d$
 has
been extensively studied, see \cite{DeGuMa} and the references contained
therein. But the sharp constant for Lebesgue measure is known only in
dimension 1, cf. \cite{Me}.

We show (cf. Theorem  \ref{kiss})
that $L(\mathbb{R}^d, \| \cdot\|)$,  the Besicovitch constant of $(\mathbb{R}^d, \| \cdot\|)$,
equals  the maximum number of 
 unit
balls that can touch a central unit ball without
touching each other (the so called strict Hadwiger number) so in particular, $L(\mathbb{R}^d, \| \cdot\|)$
 is bounded above by the translative kissing number
of $(\mathbb{R}^d, \| \cdot\|)$. This allows us to apply the extensive literature on kissing numbers to maximal function inequalities, thereby obtaining substantial improvements regarding the previously known bounds: on the plane, the uniform bounds are always either 4 or 5, depending on whether the unit ball of the given norm is a parallelogram or not.
 When 
$\| \cdot\| =  \| \cdot\|_\infty$,  we have that $L(\mathbb{R}^d, \| \cdot\|_\infty) = 2^d$ for every dimension $d$;
 furthermore, there is a locally finite Borel measure $\mu$ on $(\mathbb{R}^d, \| \cdot\|_\infty)$
 for which $\|M_\mu \|_{L^1\to L^{1, \infty}} = 2^d$, so
 this bound is attained.
 From the available information regarding kissing numbers for euclidean balls, we obtain the
sharp bound 
 $L(\mathbb{R}^3, \| \cdot\|_2) = 12$; the constant $L(\mathbb{R}^4, \| \cdot\|_2)$ is either
 $22, 23$ or $24$; for arbitrary $d$, we have the asymptotic 
 estimates
$$
(1 + o(1)) 1.1547^{ d} \le L(\mathbb{R}^d, \| \cdot\|_2)  \le 1.3205^{(1 + o(1)) d}.
$$
We remark that  for $d \gg 1$, these estimates are distinctly smaller than the bounds $2^d$ holding for cubes.

Motivated by a question of Prof. Przemys\l{}aw G\'orka (personal communication), in the last section we show that averaging operators are of strong type (1,1) for arbitrary, locally finite $\tau$-additive Borel measures $\mu$ on a metric space $X$, with bounds   independent of 
$\mu$ and of $r$, if and only if $X$ satisfies a weaker version of the Besicovitch  intersection property, called here the equal radius Besicovitch  intersection property, cf. Definition \ref{BIP} for the precise statement. It follows from the preceding results that if we have uniform weak type $(p,p)$ bounds for the centered maximal operator and some
$1 \le p < \infty$, then averaging operators are uniformly bounded on $L^1$.

I am indebted to Prof. Javier P\'erez L\'azaro for his careful reading of
this paper, as well as several useful suggestions, and to Eduardo Tablate Vila, for simplifying the proof of Lemma \ref{simple}.

\section {Definitions and general results} 

We will use $B^{o}(x,r) := \{y\in X: d(x,y) < r\}$ to denote metrically open balls, 
 and 
$B^{cl}(x,r) := \{y\in X: d(x,y) \le r\}$ to refer to metrically closed balls;
open and closed will
always be understood  in the metric (not the topological) sense. 
 If we do not want to specify whether balls are open or  closed,
we write $B(x,r)$. But when we utilize $B(x,r)$,  all balls are taken to be of the same kind, i.e., all open or all closed. Also, whenever we speak of balls,
we assume that suitable centers and radii have  been chosen (recall that in
general neither centers nor radii are unique). 

\begin{definition}  \label{tau} Let $(X, d)$ be a metric space. A Borel measure $\mu$ 
	on $X$ is   {\em $\tau$-additive} or {\em $\tau$-smooth}, if for every
	collection  $\{O_\alpha : \alpha \in \Lambda\}$
	of  open sets, we have
	$$
	\mu (\cup_\alpha O_\alpha) = \sup_{\mathcal{F}} \mu(\cup_{i=1}^n O_{\alpha_i}),
	$$
	where the supremum is taken over all finite subcollections $\mathcal{F} = \{O_{\alpha_1}, \dots, O_{\alpha_n} \}$
	of  $\{O_\alpha : \alpha \in \Lambda\}$.
	If $\mu$  assigns finite measure
	to bounded Borel sets, we say it is {\em locally finite}.
	Finally, we call $(X, d, \mu)$  a {\em metric measure space} if
	$\mu$ is a  $\tau$-additive, locally finite  Borel measure on the metric space $(X, d)$. 
	\end{definition}

The preceding definition includes all locally finite Borel measures on 
separable metric spaces and all Radon measures on arbitrary metric spaces, 
so it is more general than other
commonly used
definitions, cf. \cite{HKST} for instance. From now on we always suppose that measures are locally finite, 
not identically zero, and  that metric spaces
have at least two points. The condition of local finiteness excludes some natural measures, such as the one on $\R$ given by the density $d\mu(x) = |x|^{-1} dx$; as a matter of fact, the weak type (1,1) theory can be carried out without this assumption, since the average of an $L^1$ function over a ball of infinite measure is zero; but the $L^p$ theory fails for $p > 1$, for it may happen that the maximal function is not well defined almost everywhere.

\begin{definition}\label{maxfun} Let $(X, d, \mu)$ be a metric measure space and let $g$ be  a locally integrable function 
	on $X$. For any subset $S \subset (0,\infty)$, the {\em localized  centered Hardy-Littlewood maximal operator} $M_{S, \mu}$ is given by 
	
	\begin{equation}\label{lHLMFc}
	M_{S, \mu} g(x) := \sup _{\{r \in S : 0 < \mu (B(x, r)) \}} \frac{1}{\mu
		(B(x, r))} \int _{B(x, r)} \vert g\vert d\mu.
	\end{equation}
	Taking $S = (0, \infty)$, we obtain
	the {\em centered maximal operator } 
	$M_{\mu} := 	M_{(0, \infty), \mu} $. 
\end{definition}

 When the radii belong to an open set $S$, 
 by approximation it does not matter in the definition whether one takes the balls
$B(x,r)$  to be open or  closed.
 We
will utilize  the same notation for the maximal operators, specifying which
kind of balls we use whenever needed. 
Also, we often simplify notation by eliminating subscripts when the
meaning is clear from the context. For instance, if only one measure $\mu$ is being
considered, we may write $M$ instead of $M_\mu$.
We use  $\|M_\mu \|_{L^p\to L^{p, \infty}}$ to denote  the  weak type $(p,p)$ ``norm" of
$M_\mu$, and  $\|M_\mu \|_{L^p\to L^{p}}$ to denote its operator norm on $L^p$.

The    Besicovitch intersection property  
appears in \cite{LeRi}, where it is called the weak  Besicovitch
covering property. Our change in terminology is motivated by the fact that  this 
 property  says nothing about sets to be  covered; instead,
given a Besicovitch
family, it controls the
cardinality of the intersections at any given point.

Call a  Besicovitch family  $\mathcal{C}$ {\em intersecting} if 
$\cap \mathcal{C} \ne \emptyset$.
In  \cite{LeRi} the presentation  is local: the space $(X, d)$ has the   Besicovitch
		intersection property with constant $L$,  if there exists an integer $L\ge 1$
	such that for every intersecting Besicovitch family  $\mathcal{C}$, 
	the cardinality of $\mathcal{C}$ is bounded by $L$.

\begin{remark} To see the  equivalence of both formulations, 
	just note that  given any Besicovitch
family $\mathcal{C}$ and any  $z$ with
	$ 
	\sum_{B(x, r) \in \mathcal{C}}   \mathbf{1}_{B(x, r)} (z) > 0$,
	the set 
	$ 
	\{B(x, r) \in \mathcal{C}: z \in B(x, r)\}$	is an intersecting  Besicovitch
	family.
\end{remark}

\begin{proposition} \label{openclosed} A metric space $(X, d)$ has the  Besicovitch
		intersection property with constant $L$  for collections of open balls, if and only if 
		 has the  Besicovitch
		 intersection property for collections of closed  balls, with the same constant.
		\end{proposition}
		
\begin{proof} Denote by $L^o$ and $L^c$ the lowest constants
for collections of open balls and for collections of closed balls, respectively. 
Suppose first that  $L^o < \infty$. Let $\mathcal{C}$ be an intersecting  
 Besicovitch
family   of closed balls, and  select any finite subcollection $\{B^{cl} (x_1, r_1), \dots , B^{cl} (x_N, r_N)\}$. It is enough to
prove that $N \le L^o$. Let 
$t_i := \min\{d(x_j, x_i) : 1 \le j \le N,
j \ne i \}$. Since $t_i > r_i$, it follows that 
$\{B^o (x_1, t_1), \dots , B^o (x_N, t_N)\}$ is an intersecting Besicovitch family
of open balls,
so $N \le L^o$.

Suppose next that  $L^c < \infty$, and let $\mathcal{C}$ be an intersecting
Besicovitch
family   of open balls. Select $y \in \cap \mathcal{C}$, and
replace each ball $B^{o}(x,r) \in \mathcal{C}$ with the closed
ball $B^{cl}(x, d(x, y)) \subset B^{o}(x, r)$. The collection 
$\mathcal{C}^\prime$ so obtained is an intersecting Besicovitch family 
of closed balls, so its cardinality is bounded by $L^c$.
	\end{proof}

The following is a restatement of our main result:

\begin{theorem}\label{Borelequiv} Let $(X, d)$ be a  
	metric  space.  The following are equivalent:
	
	1)  $(X, d)$ has the  Besicovitch intersection property with constant $L$.
	
	2) For every $\tau$-additive, locally finite  Borel measure $\mu$ on $X$, the centered maximal operator
	associated to $\mu$ satisfies 
	$\|M_{\mu}\|_{L^1  \to L^{1,\infty}} \le 
	L$.	
\end{theorem}

\begin{proof} By Theorem \ref{discretization} below
it is enough to prove the result for weighted finite sums of Dirac deltas.
This is done in  Lemma
\ref{discreteequiv}.
\end{proof}

\begin{lemma}\label{discreteequiv} Let $(X, d)$ be a  
	metric  space.  The following are equivalent:
	
	1)  $(X, d)$ has the  Besicovitch intersection property with constant $L$.
	
	2)  For every  finite  weighted sum of Dirac deltas $\mu := \sum_{i = 1}^N c_i \delta_{x_i}$,
	 the centered maximal operator  satisfies 
	$\|M_{\mu}\|_{L^1  \to L^{1,\infty}}  \le 
	L$.
\end{lemma}

\begin{proof}  First  we show that  1) $\implies$ 2). 
	Let $\mu = \sum_{i = 1}^N c_i \delta_{x_i}$, where 
	$0 < c_i < \infty$. Let $0 \le f\in L^1(\mu)$ have norm $\|f\|_1 > 0$, and let $t > 0$ be such that $\mu \{M_\mu f > t\} > 0$. 
	For each $x_i$ with $M_\mu f(x_i) > t$, select $r_i > 0$ such
	that $ t \mu B(x_i, r_i) < \int_{B(x_i, r_i) } f d\mu$.

	We reorder this finite collection of balls by non-increasing radii;
	to avoid more subscripts, 
	we also	relabel   the chosen balls as $B(y_1, s_1), \dots, B(y_J, s_J)$, 
	so $s_i \ge s_{i + 1}$ and 
	 $\{y_1, \dots, y_J\}$ is just a permutation of 
$\{M_\mu f > t\} \cap \{x_1, \dots, x_N\}$.
Then we apply the standard selection procedure:
	let $B(y_{i_1} , s_{i_1}) := B(y_1, s_1)$ be the ball with largest
	radius, let 
	$B(y_{i_2} , s_{i_2})$ be the first ball in the list
	with 
	$ y_{i_2} \notin B(y_1, s_1)$, and supposing that 
	$B(y_{i_1} , s_{i_1}), \dots , B(y_{i_k} , s_{i_k}) $
	have been chosen, if all the centers $y_i$ have already
	been covered the process stops; otherwise, we let
	$B(y_{i_{k + 1}} , s_{i_{k + 1}})$ be the first ball in the
	list with 	$ y_{i_{k + 1}} \notin \cup_{i = 1}^k B(y_j, s_j)$.

	In this way, we obtain a Besicovitch family  
	$\mathcal{C}^\prime = \{B(y_{i_1} , s_{i_1} ), \dots, B(y_{i_I} , s_{i_I} )\}$
	that covers the set $\{y_1, \dots, y_J\}$. 
	By the  Besicovitch intersection property,
	$ 
	\sum_{B(y, s) \in \mathcal{C}^\prime}   \mathbf{1}_{B(y, s)} \le L$,
	and since 
 $\mu \left(\{M_\mu f > t\} \setminus \{y_1, \dots, y_J\}\right) = 0, 
	$ we have 
	$$
	\mu  \{M_{\mu} f > t\}
	\le 
	\mu ( \cup \mathcal{C}^\prime)
	\le
	\sum_{B(y, s) \in \mathcal{C}^\prime}\mu B (y,  s) 
	$$
	$$
	<  
	\sum_{B(y, s) \in \mathcal{C}^\prime}  \frac{1}{t } \int 
	\mathbf{1}_{B(y,  s)}  \ f  \  d\mu
	=
	\frac{1}{t } \ \int 
	\left(\sum_{B(y, s) \in \mathcal{C}^\prime} \mathbf{1}_{B(y,  s)} \right) \ f  \  d\mu
	\le
	\frac{L}{t } \ \int 
	\ f  \  d\mu.
	$$
	
	For  
	2) $\implies$ 1), we prove that
	if  $\mathcal{C}$  is an intersecting   Besicovitch family in $(X,d)$ of
	cardinality $ > L$, then there exists a discrete measure $\mu_c$ with 
	finite support, for which 
	$\|M_{\mu_c}\|_{L^1-L^{1,\infty}} >
	L$. 
	We may suppose
	that $\mathcal{C} = \{B(x_1, r_1), \dots ,B(x_{L + 1}, r_{L + 1})\}$
	by throwing away some balls if needed.
	Let $y \in \cap \mathcal{C}$, and for $0 < c \ll 1$, define 
	$\mu_c := c \delta_y + \sum_{i= 1}^{L + 1} \delta_{x_i}$. Set
	$f_c = c^{-1} \mathbf{1}_{\{y\}}$. Then $\|f_c\|_1 = 1$ and
	for $1 \le i \le L + 1$, $M_{\mu_c} f_c (x_i) \ge 1/(1 + c)$. Taking $y$ into account, we get
	$$
	\mu_c \{M_{\mu_c} f_c \ge  1/(1 + c)\} = L + 1 + c,
	$$
	 so for 
	$c$ small enough, 	$\mu_c \left( \{M_{\mu_c} f_c \ge  1/(1 + c)\}\right)/(1 + c) > L$.
\end{proof}

The ``local discretization of measures" Theorem \ref{discretization} below,  states that  uniform bounds on weighted finite sums of Dirac deltas, extend to uniform bounds on arbitrary ($\tau$-additive
locally finite) Borel measures. 
Note that $\|M_{\mu}\|_{L^1  \to L^{1,\infty}} $ is not assumed to be finite
in either  Lemma \ref{simple} or in Theorem \ref{discretization}. 

Next we state 
three lemmas, some parts  of which are well known in the absence of localization.
The first one follows by a standard approximation argument, so the proof is omitted.

\begin{lemma}\label{loc} Let $(X, d, \mu)$ be a  
	metric  measure space.  For $0 \le  s < S \le \infty$, the values of the
 localized centered maximal operator
 	$M_{(s,S), \mu}$ are independent  of whether $M_{(s,S), \mu}$ is defined
 	using open or closed balls.
\end{lemma}

\begin{lemma}\label{simple} Let $(X, d)$ be a  
	metric  space.  If
	there is a locally finite  Borel measure $\mu$ such that 
	$\|M_{\mu}\|_{L^1-L^{1,\infty}} >  
	L$, then  there exist a $T > 0$,  a ball $B^{o}(y, R)$, and a simple function $f$
	vanishing outside $B^{o}(y, R + T)$, 
	such that 
	$$
	\mu \left(B^{o}(y, R) \cap  \{M_{ \mu} f > t\}\right) > 
	\frac{L}{t } \ \int 
	\ f  \  d\mu.
	$$
\end{lemma}

\begin{proof}  The argument proceeds by using several standard reductions to simpler cases. 	
	If $\|M_{\mu}\|_{L^1-L^{1,\infty}} >
	L$, then we can select $0 \le h\in L^1(\mu)$, $R > 0 $, $t > 0$, and $y \in X$  such that 
	$$
	\mu \left(B^{o}(y, R) \cap  \{M_{ \mu}  h > t\}\right) > 
	\frac{L}{t } \ \int 
	\ h  \  d\mu.
	$$
	An additional approximation argument tells us that for some $T \gg 1$, 
	$M_{ \mu}$ can be replaced in the above inequality by its localized
	variant  $M_{(0, T), \mu}$. 
	
	Clearly,  we only need to consider what happens inside $B^{o}(y, R +  T) $
	to determine the behavior of $M_{(0, T), \mu}$ in $B^{o}(y, R) $, so
	there
	is no loss in assuming that $h$ vanishes identically outside 
	$B^{o}(y, R + T)$.  
	
	Next we show that $h$ can be suitably approximated by a simple function $f$, that is,
	of the form $f = \sum_{i = 1}^J c_i \mathbf{1}_{S_i}$, where the
	$S_i$ are disjoint Borel sets contained in 
	$B^{o}(y, R + T)$, and the coefficients $c_i$ are strictly positive.

	If $h$ is bounded then the result is clear,
	for given any $\varepsilon > 0$ we can always find a simple function
	$f = f \mathbf{1}_{B^{o}(y, R + T)}$ such that $0 \le h \le f $ and $\| f \|_1 < (1 + \varepsilon) \|h\|_1.$
	If $h$ is unbounded,  we choose $H \gg 1$ so that the truncation
	$h \wedge H := \min\{h, H \}$ is sufficiently close to $h$
	(and then  we are back to the previous case) as follows. Set $E := 	B^{o}(y, R) \cap  \{M_{ \mu} h > t\}$ and note that for each $x\in E$ there exists an $r_x > 0$ such that 
$$
t 
<
\frac{1}{ \mu B(x,  r_{x}) } \int_{B (x, r_x) }  h \  d\mu.
$$
By the Monotone Convergence Theorem, there is an $n_x \in \mathbb{N}$ such that 
$$
t 
<
\frac{1}{ \mu B (x,  r_{x}) } \int_{B  (x, r_x) }  h \wedge n_x  \ d\mu.
$$
Let $E_n := E \cap  \{M_{ \mu} (h \wedge n) > t\}$. Then each $E_n$ is measurable, $E_n \subset E_{n + 1}$, and $E = \cup_{n = 1}^\infty E_n$, so 
$$
\mu E = \lim_n \mu E_n > 
	\frac{L}{t } \ \int 
\ h  \  d\mu.
$$
Thus, there exists an $H$ such that  
$$
\mu E_H > 
\frac{L}{t } \ \int 
\ h  \  d\mu
\ge
\frac{L}{t } \ \int 
\ h \wedge H \  d\mu.
$$
\end{proof}
	
	From now on, it will be more convenient to use closed balls. 

\begin{lemma}\label{open} Let $(X, d, \mu)$ be a  
	metric  measure space.  
	For $0 \le s < S \le \infty$, $0 \le f \in L^1(\mu)$,  $t > 0$ and 
	$u \ge 0$, the 
	set 	
	\begin{equation} \label{O}
	 O_{t , u} := \left\{ x \in X : \exists r\in (s,S) \mbox{   \  with  \ }
	 \frac{1}{ \mu B^{cl}(x, r)} \int_{B^{cl}(x, r) } f d\mu > t
	 \mbox{   \  and  \ }
	\mu  B^{cl}(x, r) > u
	\right\}
	\end{equation}
	is open.
\end{lemma}

\begin{proof} If $ O_{t , u} $ is empty there is nothing to show, so suppose 
otherwise. Choose $x \in  O_{t , u} $ and $r\in (s,S)$ such that  
$$
\frac{1}{ \mu B^{cl}(x, r)} \int_{B^{cl}(x, r) } f d\mu > t
 \mbox{   \  and  \ }
\mu  B^{cl}(x, r) > u.
$$ 
Fix $\varepsilon >0$ with
$$
\frac{1}{ \mu B^{cl}(x, r)} \int_{B^{cl}(x, r) } f d\mu > (1 + \varepsilon) t.
$$
Select
$0 < \delta < r$ with $r + \delta < S$ and
$
\mu B^{cl}(x, r + \delta ) < (1 + \varepsilon)  \mu B^{cl}(x, r) $ 
(here we use that
 balls are metrically closed).
 Let $y \in   B^{o}(x, \delta /2)$. Then 
 $$
B^{cl}(x, r) \subset  B^{cl}(y, r +  \delta /2) 
\subset  B^{cl}(x, r + \delta ),
$$ 
so $\mu  B^{cl}(y, r +  \delta /2) > u$ and 
$$
(1 + \varepsilon) t  
< 
\frac{1}{ \mu B^{cl}(x, r)} \int_{B^{cl}(x, r) } f d\mu
$$
$$
\le 
\frac{(1 + \varepsilon)}{ \mu B^{cl}(x, r + \delta )} \int_{B^{cl}(y, r + \delta /2) } f d\mu
\le
\frac{1 + \varepsilon}{ \mu B^{cl}(y, r + \delta /2)} 
\int_{ B^{cl}(y, r + \delta  /2)} f d\mu.
$$
\end{proof}

\begin{theorem}\label{discretization} Let $(X, d)$ be a  
	metric  space.  If
	there is a $\tau$-additive, locally finite  Borel measure $\mu$ such that 
$\|M_{\mu}\|_{L^1-L^{1,\infty}} >  
	L$, then  there is a discrete, finite  Borel measure $\nu$ with finite support
	in $X$, for which 
	$\|M_{\nu}\|_{L^1-L^{1,\infty}} >
	L$.	
\end{theorem}

\begin{proof} By the preceding lemmas, it is enough to show that
	given  $\varepsilon > 0$, $R , T > 0 $, $t > 0$,  $y \in X$,  
	and a simple function $0 \le f\in L^1(\mu)$ vanishing outside
	$B^{o}(y, R + T)$, it is possible to select a finite discrete measure
	$\nu := \sum_{i = i}^m a_i\delta_{w_i}$, with $a_i > 0$, 
	such that $\| f \|_{L^1(\nu)} =\| f \|_{L^1(\mu)}$ and
	$$
	\mu\left( B^{o}(y, R) \cap  \{M_{(0, T), \mu} f > t\}\right)
	\le
	(1 + \varepsilon )  \nu  \left(B^{o}(y, R) \cap  \{M_{(0, T), \nu} f >  t/(1 + \varepsilon) \}\right),
	$$
	where the localized maximal operators are
	defined using metrically closed balls.

	For each  
	$x \in O_t := B^{o}(y, R) \cap  \{M_{(0, T), \mu} f > t\}$,  select $0 < r_x <  T$ such
	that 
	$$
	t \mu B^{cl}(x, r_x) < \int_{B^{cl}(x, r_x) } f d\mu,
	$$
	and choose
	$0 < \delta_x < T$ so
	that 
	$
	\mu B^{cl}(x, r_x +  \delta_x ) < (1 + \varepsilon) \mu  B^{cl}(x, r_x) $. It follows from Lemma \ref{open} that $O_t$ is open; we select very small radii 
	$0 < s_x < \min\{r_x,  \delta_x \}/2$
	so that for each $x$,  we have $B^{cl}(x, s_x) \subset O_t$. Since  
	$$
	O_t
	= 
	\cup \left\{B^{o}(x, s_x) : x \in O_t\right\},$$
	by $\tau$-additivity we can pick a finite collection of
	centers $x_1, \dots, x_n$ in
	$O_t$ such that
	$$
	\mu O_t < (1 + \varepsilon) 
	\mu \cup_{i= 1}^n B^{o}(x_i, s_{x_i}).
	$$
	Let	$f = \sum_{i = 1}^J c_i \mathbf{1}_{S_i}$, where the
	$S_i$ are disjoint Borel subsets of  
	$B^{o}(y, R + T)$, and the coefficients $c_i$ are strictly positive.
	The next step consists in defining a suitable finite subalgebra $\mathcal{A}$ on 
	$B^{o}(y, R + T)$. We let $\mathcal{A}$ 
	be generated by the sets $S_i$ defining $f$, for $1 \le i \le J$,
	together with  $B^{o}(y, R), B^{o}(y, R + T),  O_t$, and the finite collection of balls 
	$ B^{cl}(x_i, u_i)$, where $u_i$ takes each of the three values
	$s_{x_i}$,  $r_{x_i}$ and $r_{x_i} +  \delta_{x_i}$, for
	$1 \le i \le n$.
	
	Given $z \in B^{o}(y, R + T)$, let 
	$P_z := \cap \{A \in \mathcal{A} : z \in A\}$. The
	sets $P_z$ are the atoms of  $\mathcal{A}$, 
	so  they yield a finite partition
	$\{P_1, \dots , P_m\}$ of $B^{o}(y, R + T)$ by non-empty measurable sets.
	Also, we may assume that for each $1 \le i \le m$, 
	$\mu P_i > 0$, for otherwise we simply disregard a finite
	number of sets of measure zero.  Since
	each $P_i$ 
	cannot be split into smaller sets belonging to $\mathcal{A}$, the
	value of any measure on $\mathcal{A}$ is completely determined by its
	value on these atoms.
	Choose representatives $w_i \in P_i$, $1 \le i \le m$, and set
	$\nu \{w_i\} = \mu P_i$. Then $\nu  = \mu$ on $\mathcal{A}$. Furthermore,   $\nu$ is  defined for  all subsets of $X$, since it is
	discrete measure.
	
	By the $\mathcal{A}$-measurability of $f$ and of the balls
	$B^{cl}(x_i, r_{x_i})$, $1 \le i \le n$, we have that
	$\nu B^{cl} (x_i, r_{x_i}) = \mu B^{cl} (x_i, r_{x_i})$ and
	$$
	\int_{B^{cl}(w_i, r_i) } f d\nu = 
	\int_{B^{cl}(w_i, r_i) } f d\mu.$$
	
	We claim that for $\nu$-almost very point in  
	$\cup_{i= 1}^n B^{cl}(x_i, s_{x_i})$, 
	the inequality
	$(1 + \varepsilon) M_{(0, T), \nu} f > t $ holds. This yields
	the result, since then
	$$
	 \nu  \left(B^{o}(y, R) \cap  \{M_{(0, T), \nu} f >  t/(1 + \varepsilon) \}\right)
	\ge
	\nu \cup_{i= 1}^n B^{cl}(x_i, s_{x_i}) 
	$$
	$$
	=
	\mu \cup_{i= 1}^n B^{cl}(x_i, s_{x_i}) 
	>
	\frac{\mu O_t}{1 +\varepsilon}
	>
	\frac{L}{(1 +\varepsilon) t } \ \int 
	\ f  \  d\mu
	=
	\frac{L}{(1 +\varepsilon) t } \ \int 
	\ f  \  d\nu.
	$$
	To see why the claim is true, recall that the representatives 
	$w_i$  constitute the support of $\nu$. 
	Choose any  $w_j \in  \cup_{i= 1}^n B^{cl}(x_i, s_{x_i})$. For some $1 \le k \le n$, 
	we have $w_j \in B^{cl}(x_k, s_{x_k})$.
	Now 	 
	$$
	B^{cl}(x_k, r_{x_k}) \subset  B^{cl}(w_j, s_{x_k} +  r_{x_k}) 
	\subset  B^{cl}(x_k, r_{x_k} +   \delta_{x_k}),
	$$ 
	so
	$$
	t  < 
	\frac{1}{ \mu B^{cl}(x_k, r_{x_k})} \int_{B^{cl}(x_k, r_{x_k}) } f d\mu
	=
	\frac{1}{ \nu B^{cl}(x_k, r_{x_k})} \int_{B^{cl}(x_k, r_{x_k}) } f d\nu
	$$
	$$
	\le 
	\frac{1}{ \nu B^{cl}(x_k, r_{x_k})} \int_{B^{cl}(w_j, s_{x_k} + r_{x_k}) } f d\nu
	\le
	\frac{1 + \varepsilon}{ \nu B^{cl}(w_j, s_{x_k} + r_{x_k}) } 
	\int_{B^{cl}(w_j, s_{x_k} + r_{x_k}) } f d\nu.
	$$
\end{proof}

A modification of the  proof of 2) $\implies$ 1)  in Lemma
\ref{discreteequiv},  shows that for any $p \in (1, \infty)$, the uniform weak
type $(p,p)$ already implies the  Besicovitch intersection property. 
Recall that the floor function $\lfloor x \rfloor$ denotes the integer part of $x$.

\begin{theorem}\label{extrapolation} Let $(X, d)$ be a  
	metric  space. Each of the following statements implies the next:
	
1)	There exist a $p$ with $1 < p < \infty$ and an integer $N\ge 1$, such that
for every discrete, finite  Borel measure $\mu$ with finite support
in $X$, the centered maximal operator
associated to $\mu$ satisfies 
	$\|M_{\mu}\|_{L^p\to L^{p,\infty}} \le N$.
	
2) The space $(X, d)$ has the  Besicovitch intersection property with constant at most
$\lfloor p^p (p - 1)^{(1-p)} N^p - p \rfloor + 1$.
	
	3)  For every $\tau$-additive, locally finite  Borel measure $\mu$ on $X$, the centered maximal operator
	associated to $\mu$ satisfies 
	$\|M_{\mu}\|_{L^1-L^{1,\infty}} \le 
	\lfloor p^p (p - 1)^{(1-p)} N^p - p \rfloor + 1$.	
	\end{theorem}

\begin{proof} Recall that  2) $\implies$ 3) has already been 
	proved in Lemma \ref{discreteequiv}, with a different expression for the constant.  Regarding
1) $\implies$ 2),  let $q = p/(p - 1)$ be the dual exponent of $p$,  let $\mathcal{C} = \{B(x_1, r_1), \dots ,B(x_{J}, r_{J})\}$ be an intersecting  Besicovitch family  in $(X,d)$, and let $y \in \cap \mathcal{C}$. Define, for $ c > 0$, the measure
$\mu_c :=  c \delta_y + \sum_{i= 1}^{J} \delta_{x_i}$, and
recall that for every $\alpha > 0$, 
\begin{equation}\label{weaktypepM}
\mu_c (\{M_{\mu_c} f \ge \alpha\}) \le \left(\frac{\|M_{\mu_c}\|_{L^p\to L^{p,\infty}} \|f\|_{L^p}}{\alpha}\right)^p.
\end{equation}
Set
$f = \mathbf{1}_{\{y\}}$; then $\|f\|_{L^p(\mu_c)} = c^{1/p}$. For $1 \le i \le J$, we have $M_{\mu_c} f (x_i) \ge c/(1 + c)$. 
Thus,
$\mu_c \{M_{\mu_c} f \ge  c/(1 + c)\} = J  + c$ (taking $y$ into account) so 
with $\alpha = c/(1 + c)$, we have 
\begin{equation}\label{weaktypepM1}
J + c
\le
\left(\frac{\|M_{\mu_c}\|_{L^p\to L^{p,\infty}} (1 + c)}{c^{1/q}}\right)^p.
\end{equation}
Maximizing $g(c) = c^{1/q}/(1 + c)$ we get $c = p - 1$
and $g(p - 1) = (p - 1)^{(p - 1)/p} p^{-1}$,  so
\begin{equation}\label{weaktypepM11}
J +  p - 1
\le
\|M_{\mu_c}\|_{L^p\to L^{p,\infty}}^p \  p^p \ (p - 1)^{1 - p} 
\le
N^p \ p^p \ (p - 1)^{1 - p}.
\end{equation}
\end{proof}

\begin{remark} The extrapolation result  1) $\implies$ 3) tells us that for any
$1 < p < \infty$, uniform weak type $(p,p)$ bounds of size $N$ entail  uniform weak type $(1,1)$ bounds of size less than
$p^p (p - 1)^{(1-p)} N^p$. Once we have these, 
by interpolation we get the following strong
type $(p,p)$ bounds:
$$
 \|M_{\mu}\|_{L^p\to  L^{p}} 
 \le  
  \frac{p^2 N}{(p -1)^{2  - 1/p}},
 $$
(cf. \cite[p. 42, Exercise 1.3.3 (a)]{Gra}) so the factor
$ 
 p^2  (p -1)^{- 2  + 1/p}
 $
 bounds from above the ratio between the uniform strong and uniform weak $(p,p)$ bounds,
 for every $p \in (1, \infty)$. 
 
 Note however that this bound might considerably 
 overestimate the actual ratio, since in general interpolation will
 not yield the best possible  constants, and occasionally it may yield 
 bounds that are very far from optimality; for instance, it is known 
 that 
for  $p = 2$, for Lebesgue measure in $\mathbb{R}^d$, and for balls defined by an arbitrary
norm,
optimal constants are uniformly bounded by 140 in every dimension
(see \cite[Theorem 5.2]{DeGuMa}).
 For cubes (balls with respect to the $\ell_\infty$ norm), J. Bourgain 
 has proved that  dimension independent
bounds hold for every $p > 1$, cf.
 \cite{Bou1}. However, for cubes it is also known that the weak type (1,1) constants diverge to
infinity with the dimension  (cf. \cite{Al}) and thus, for every $p \in (1, \infty)$, so do the bounds obtained by interpolation.
\end{remark}

\section{Consequences for $\mathbb{R}^d$}

Again  we take balls to be closed. Recall that  $L(\mathbb{R}^d, \| \cdot\|)$  denotes
the Besicovitch constant of $(\mathbb{R}^d , \|\cdot\|)$.
The  definition of strict Hadwiger number  comes from \cite[p. 123]{MaSw}, but this notion
had been used before.

\begin{definition} Let $\|\cdot\|$ be any norm on 
 $\mathbb{R}^d$. The  {\em  Hadwiger number} or {\em  translative kissing number} 
 $H(d, \| \cdot\|)$, is the maximum number of 
  translates of the closed unit ball $B^{cl} (0, 1)$ that can touch $B^{cl} (0, 1)$ without overlapping,
  i.e., all the translates have disjoint interiors.
The  {\em strict Hadwiger number} $H^*(d, \| \cdot\|)$ is the maximum number of 
  translates of the closed unit ball $B^{cl} (0, 1)$ that can touch $B^{cl} (0, 1)$ without touching each other,
  that is, all the translates are disjoint.
 A {\em  spherical code}  is a finite set of unit vectors.
	\end{definition}

\begin{theorem} \label{kiss}  Let $\|\cdot\|$ be any norm on 
 $\mathbb{R}^d$. The Besicovitch constant of $(\mathbb{R}^d , \|\cdot\|)$ equals
its strict Hadwiger number, i.e., $L(\mathbb{R}^d, \| \cdot\|) = H^*(d, \| \cdot\|)$. 
\end{theorem}

\begin{proof}  To see that  $L(\mathbb{R}^d, \| \cdot\|) \le  H^*(d, \| \cdot\|)$, let 
$\mathcal{C} := \{B^{cl} (x_1, r_1),  \dots,  B^{cl} (x_n, r_n) \}$  be an intersecting 
Besicovitch family in 
	 $\mathbb{R}^d$ of maximal cardinality. Choose $y \in \cap \mathcal{C}$, 
	 and let $r_y := \min\{ \|x_1 - y\| ,  \dots,  \|x_n - y\|\}$.
By a dilation and a translation, if needed, we can assume that
$y = 0$ and $r_y = 1$. We claim that all the balls in
$\mathcal{C}^\prime := \{B^{cl} (2 x_1/\|x_1\|, 1),  \dots,  B^{cl} (2 x_n/\|x_n\|, 1) \}$ are disjoint,
and clearly they touch $B^{cl} (0, 1)$, so $n \le  H^*(d, \| \cdot\|)$.  To check the claim 
it is enough to verify that any two centers $2 x_i/\|x_i\|$ and $2 x_j/\|x_j\|$ are at distance
$ > 2$, or equivalently, that any two vectors in the spherical code
$ \{ x_1/\|x_1\|,  \dots,  x_n/\|x_n\|\}$  are at distance
$ > 1$.
So choose a pair of centers  $x_i$ and $x_j$ of balls from $\mathcal{C}$, with, say,  $\|x_i\| \ge  \|x_j\|$. 
Since $\|x_i - x_j\|  > \|x_i\|$, using the lower bound for the angular distances from   \cite[Corollary 1.2]{Ma}, we get
\begin{equation*}
\left\|\frac{x_i}{\|x_i\|}-\frac{x_j}{\|x_j\|}\right\|
\ge
\frac{\|x_i - x_j\|  - \left| \|x_i\| - \|x_j\|\right| }{\min\left\{\|x_i\|, \|x_j\|\right\} }
= \frac{\|x_i - x_j\|  - \|x_i\| + \|x_j\|}{ \|x_j\| } 
>
1.
\end{equation*}
 
  For the other direction,  each
set of unit vectors  $S$ satisfying $\| x - y \| > 1$   for all  $x, y \in S$ with $x \ne y$,
defines an intersecting Besicovitch family $\{B^{cl} (x, 1) : x \in S\}$, so
 $L(\mathbb{R}^d, \| \cdot\|) \ge H^*(d, \| \cdot\|)$. 
 \end{proof}

In $\mathbb{R}^d$  there is ``plenty of  room", so it is possible to
construct a measure $\mu$ for which 
the supremum  is attained.

\begin{theorem} \label{kiss}  Let $\|\cdot\|$ be any norm on 
 $\mathbb{R}^d$. Then there exists a discrete measure $\mu$ 
 such that $\|M_{\mu}\|_{L^1\to L^{1,\infty}}  = L(\mathbb{R}^d, \| \cdot\|)$.
\end{theorem}

\begin{proof}  Given $\| \cdot\|$,  by rescaling if needed we may assume that  $\| e_1 \| =1$. 
Then  we argue as in the
proof of Lemma \ref{discreteequiv}: choose a
spherical code $\{x_1, \dots, x_N\}$ of   cardinality $H^*(d, \| \cdot\|)$,  with minimal separation strictly larger than $1$. 
For $n \ge 1$, set
$\mu_n := n^{-1} \delta_{3 n e_1} + \sum_{i= 1}^{N} \delta_{x_i + 3 n e_1}$,
and let $\mu := \sum_{n=1}^\infty \mu_n$.
\end{proof}

Note that the measure $\mu$ in the preceding result can be chosen to be finite,  by assigning
suitable weights to the measures $\mu_n$. 

The best uniform bound in one dimension  for the {\em uncentered} maximal operator
is $2$,
cf. \cite[Formula (6)]{CaFa}. Since $ L(\mathbb{R}, | \cdot |) = 2$,  
the same uniform bound holds for both the centered and the uncentered operators in dimension 1.
But already in dimension 2 (for squares and discs) the standard gaussian measure provides
an example where the uncentered maximal operator is not
of weak type $(1,1)$, cf. \cite{Sj}.

\begin{corollary} \label{2bounds} Given any norm $
	\|\cdot\|
	$ on the plane,  if the unit ball is a parallelogram then    $
	L(\mathbb{R}^2, \|\cdot\|)  = 4
	$, while 
	$
	L(\mathbb{R}^2, \|\cdot\|)  = 5
	$ 
	in every other case.
\end{corollary}

\begin{proof} This follows from the corresponding results for strict
	Hadwiger numbers, cf. \cite[Proposition 23]{Sw}.
\end{proof}

\begin{corollary} \label{infinitybounds} The sharp uniform bound for the
	centered maximal operator on $(\mathbb{R}^d , \|\cdot\|_\infty)$ 
	is $L(\mathbb{R}^d, \| \cdot\|_\infty) = 2^d$. Furthermore, the bound is attained.
\end{corollary}

\begin{proof} It is enough to check that $H^*(d, \| \cdot\|_\infty) = 2^d$, something that is both well known and
easy to see. The inequality $H^*(d, \| \cdot\|_\infty) \ge 2^d$ follows by placing translates of the unit cube touching
the central cube only at the vertices, and the other direction
follows  by noticing that any  cube
 touching the central cube must touch some vertex. A more general result can be found in \cite[Lemma 3.1]{Ta}.
 \end{proof}

Next we consider euclidean balls. In this context, the translative kissing number is just the kissing
number, $A(d, \theta)$ denotes  the
maximum number of unit vectors in $\mathbb{R}^d$ such that for any pair $x, y$ of them,
$x\cdot y \le \cos \theta$, and $A^\circ (d, \theta)$ is defined  in the same way, but requiring the 
inequality to be strict, so
$x\cdot y <  \cos \theta$. Observe that $A (d, \pi/3) = H(d, \| \cdot\|_2)$,
while $A^\circ (d, \pi/3) = H^*(d, \| \cdot\|_2)$.

\begin{corollary} \label{bounds} The sharp uniform bound for the
	centered maximal operator on $(\mathbb{R}^3 , \|\cdot\|_2 )$, is  $L(\mathbb{R}^3, \|\cdot\|_2) = 12$. Asymptotically we have
\begin{equation}\label{asym}
	(1 + o(1)) \sqrt{\frac{3 \pi}{8}}  \log {\frac{3}{2 \sqrt 2}}
\ 	d^{3/2} \  \left(\frac{2}{\sqrt{3}}\right)^d
\le L(\mathbb{R}^d, \|\cdot\|_2)
	\le
2^{0.401 (1 + o(1)) d}.
\end{equation}
\end{corollary}

\begin{proof}  For $d = 3$ it is well known
that a spherical code of maximal cardinality (12 vectors) can be obtained from the
vertices of a
regular icosahedron inscribed in the unit sphere. Since the minimal separation  between any two
vertices of the icosahedron   is strictly larger than 1, we have $ 12 = A(3, \pi / 3) =  A^\circ (3, \pi / 3)$.

 Regarding the asymptotic   bounds,  the left hand side 
in  formula (\ref{asym}) comes from \cite[Theorem 1]{JeJoPe}; up to constants, it improves
previously known bounds by a factor of $d$.
 Trivially 
 $A^\circ  (d, \pi/3) \ge  A  (d, \theta)$ for every  $\theta > \pi/3$. Using the estimates  in
  \cite[Theorem 2]{JeJoPe} for $ A  (d, \theta)$,  when
 $0 < \theta < \pi/2$, we conclude 
 that the lower bounds given in \cite[Theorem 1]{JeJoPe}
 by taking  $\theta = \pi/3$ are also lower bounds for 
 $A^\circ  (d, \pi/3)$, since for $d$ fixed all the parameters 
 in \cite[Theorem 2]{JeJoPe} depend 
 continuously on $\theta$. And the right hand side of
   (\ref{asym}) follows directly
 from the upper bounds known for $A  (d, \pi/3)$, cf. \cite[Corollary 1, p. 20]{KaLe}, or \cite[Formulas (66) and (49)]{CoSl}.
\end{proof}

\begin{remark} For $d > 3$, the exact values of $A  (d, \pi/3)$ presently  known are 
$A  (4, \pi/3) = 24$ (cf. \cite{Mu}), 
$A  (8, \pi/3) = 240$ and  $A  (24, \pi/3) = 196560$ (cf. \cite[p.12, Table 1.1]{CoSl}) but additional upper and lower
bounds can be found in the literature, cf. \cite{BaVa} for instance. Judging from \cite[Part 1]{SlHaSm},  it would appear that
$A^\circ (4, \pi / 3)  = 22 < 24 =  A (3, \pi / 3)$; however, since the minimal separation for 23 unit vectors is given
as $60.0000000^\circ$, instead of, say,  $\pi/3$, and the packings there are only claimed to be ``putatively optimal", the actual
value of $A^\circ (4, \pi / 3) $ is not clear to me.

Ignoring the terms that are not exponential in $d$, the preceding corollary  entails that 
$
(1 + o(1)) 1.1547^{ d} \le L(\mathbb{R}^d, \|\cdot\|_2)  \le 1.3205^{(1 + o(1)) d},
$
which are the bounds indicated in the introduction. Curiously, the uniform bounds satisfied by the
centered operator associated to cubes are smaller than those associated to euclidean balls in dimensions
2, 3, and 4; in dimension 8 the situation is reversed, since $256 > 240$, and  for $d\gg 1$, the
bounds associated to cubes are much larger. Strict Hadwiger numbers
for other norms have also been studied, cf. for
instance \cite{RoSa}, \cite{Sw1}, \cite{Ta}.

  Denote by $\mu_d (A) := \lambda^d (A\cap B (0,1))$
 the  Lebesgue
 measure  $\lambda^d$ restricted to the euclidean unit ball of $\mathbb{R}^d$; 
 the measures $\mu_d$ provide a concrete family where the exponential factor in the left hand side of 
 (\ref{asym}) is present: by 
 \cite[Remark 2.7]{Al1},
\begin{equation*}
 \|M_{\mu_d} \|_{L^1\to L^{1, \infty}}
  \ge 
\frac{\sqrt{\pi (d+1)}}{\sqrt 6} \left( \frac
{2}{\sqrt 3}\right)^{d}.
\end{equation*}
Regarding the maximal number $\beta_d$ of disjoint collections appearing
in the Besicovitch covering theorem in dimension $d$, it has been studied in 
\cite{Su} for euclidean balls and in \cite{FuLo} for balls
associated to general norms. Specifically, in \cite[Theorem]{Su} it is noted that $\beta_2 = 19$, 
$\beta_3 \le 87$, and asymptotically, $\beta_d$ grows exponentially with $d$, to base at least $8/\sqrt{15}$ and at
most $2.641$. 
Thus,  the bounds for the centered maximal operator using kissing numbers represent a substantial improvement
over the $\beta_d$'s.
\end{remark}

\begin{question} It would be desirable to have a better understanding of the relationship between $L(\mathbb{R}^d, \|\cdot\|_2) = A^\circ  (d, \pi/3)$  and $A(d, \pi/3)$,
a subject that hopefully will be of interest to specialists in spherical codes. For large angles this was solved in
\cite{Ra}. In particular, by \cite[Theorem 1]{Ra}, $A^\circ  (d, \pi/2) = d + 1$  and $A(d, \pi/2) = 2d $. \end{question}

 \section {Boundedness of averaging operators on $L^1$}

 Recall that 
 the complement of
 the support $(\operatorname{supp}\mu)^c := \cup \{ B^{o}(x, r): x \in X, \mu B^{o}(x,r) = 0\}$
 of a Borel  measure
 $\mu$,  is an open set, and hence measurable. 
 
  \begin{definition}\label{maxfun} Let $(X, d, \mu)$ be a metric measure space and let $g$ be  a locally integrable function 
 	on $X$. For each fixed $r > 0$ and each $x\in \operatorname{supp}\mu$, the
 	averaging operator $A_{r, \mu}$ is defined as
 	\begin{equation}\label{avop}
 	A_{r , \mu} g(x) := \frac{1}{\mu
 		(B(x, r))} \int _{B(x, r)}  g \ d\mu.
 	\end{equation}
 \end{definition}

While maximal operators are well defined everywhere, by our convention averaging operators  are defined only on the support of the  measure under consideration.
 
 \begin{definition}\label{full} Let $(X, d)$ be a metric space and let
 	$\mu$ be a locally finite Borel measure on $X$. 
 	If $\mu (X \setminus \operatorname{supp}\mu) = 0$, 
 	we say that $\mu$ has {\em full support}. 
 \end{definition}
 
 By $\tau$-additivity,  if $(X, d, \mu)$ is  a metric measure space, then 
 $\mu$ has full support,
 since the set $X \setminus \operatorname{supp}\mu $ is a union of open balls of measure zero. Thus, averaging operators are a.e. defined on metric measure spaces.

 To  specify whether balls are open or closed,
 we use $A_{r , \mu}^{o} $ and $A_{r , \mu}^{cl} $ for the corresponding operators. 
When we consider only one measure $\mu$, we often   
 write $A_{r } $ instead of $A_{r , \mu} $.

 \begin{definition} \label{conjugate} We call 
 	\begin{equation}\label{conjug}
 	a_r (y)  
 	: =
 	\int_{\operatorname{supp(\mu)}}  \frac{\mathbf{1}_{B(y,r)}(x)}{\mu B(x,r)}  \  d\mu(x) 
 	\end{equation}
 	the {\em conjugate function} to the averaging operator $A_{r}$.
 \end{definition}
 
 Note that the conjugate function $a_r$ is well defined a.e. (whenever $y$ belongs to
 the support of $\mu$). According to \cite[Theorem 3.3]{Al2}, 
 $A_r$ is bounded on
 $L^1(\mu)$ if and only if $a_r \in L^\infty (\mu)$, in which case 
 $\|A_r \|_{L^1(\mu)\to L^1(\mu)} = \|a_r\|_{L^\infty(\mu)\to L^\infty (\mu)} $. 
 We will  use $a_{r }^{o} $ and $a_{r }^{cl} $ to  specify whether balls are open or closed.

 \begin{definition}  \label{BIP}  Denote by
 	$\mathcal{E}(X,d)$ the collection of all Besicovitch families $\mathcal{C}$
 	of $(X, d)$ with the property that all balls in
 	$\mathcal{C}$ have equal radius.  The 
 	{\em equal radius Besicovitch constant}  of $(X, d)$ is
 	\begin{equation}\label{BC}
 	E (X, d) := \sup \left\{	\sum_{B(x,  r) \in \mathcal{C}} \mathbf{1}_{B(x,  r)} (y): y \in X, \ \mathcal{C} \in \mathcal{E}(X,d) \right\}.
 	\end{equation}
 	We say that $(X, d)$ has the  {\em equal radius Besicovitch Intersection Property} with constant $E (X, d)$  if $E (X, d) < \infty$.
 \end{definition}

 \begin{example} It is well known 
that  the Heisenberg groups
 	$\mathbb{H}^n$ with the
 	Kor\'any metric do not have the Besicovitch intersection property, cf.  \cite[pages 17-18]{KoRe} or
 	\cite[Lemma 4.4]{SaWh}. However, they do have the equal radius  Besicovitch intersection property, since they are geometrically doubling.
 	
 	 We mention why geometrically doubling entails $E (X, d) < \infty$.	Recall that  a metric space is {\it geometrically doubling}  if there exists a positive
 	integer $D$ such that every ball of radius $r$ can be covered with no more than $D$ balls
 	of radius $r/2$.  We call the smallest such $D$ the {\em doubling constant} of the space.  Now  $E (X, d) \le D$. To see why,  given an intersecting Besicovitch family $\mathcal{C}$ with equal radius  $r$, choose any $y\in \cap \mathcal{C}$, and note that the centers of all balls in $\mathcal{C}$ form an $r$-net in $B(y,r)$. Cover $B(y,r)$ with  $\le D$ balls of radius $r/2$. Since each such ball contains at most the center of one ball from $\mathcal{C}$, the result follows. 
 	 \end{example}

 While being geometrically doubling does not depend on whether we use open or closed balls in the definition,  the  value of the doubling constant may change. However,  $E(X,d)$ is the same for open and closed balls, as the  next proposition indicates. Since the proof is a straightforward modification of the argument given for Proposition \ref{openclosed}, we omit it.
 
 \begin{proposition}  A metric space $(X, d)$ has the equal radius  Besicovitch
 	intersection property with constant $E$  for collections of open balls, if and only if it
 	has the  equal radius  Besicovitch
 	intersection property for collections of closed  balls, with the same constant.
 \end{proposition}

 The situation regarding the existence of  strong type $(1,1)$ bounds for the averaging operators 
 $A_{r, \mu}$, uniform in both $r$ and $\mu$, is  analogous to the weak type case for the maximal operator, but with the equal radius Besicovitch intersection property replacing the  Besicovitch intersection property.

 \begin{theorem}\label{discreteequivav} Let $(X, d)$ be a  
 	metric  space.  The following are equivalent:
 	
 	1)  The space $(X, d)$ has the  equal radius Besicovitch intersection property with constant $E$.
 	
 	2) For every $r > 0$ and every $\tau$-additive, locally finite  Borel measure $\mu$ on $X$, we have
 	$\|A_{r, \mu}\|_{L^1  \to L^1} \le 
 	E$.	
 	
 	3)  For every  $r > 0$ and every finite  weighted sum of Dirac deltas $\mu := \sum_{i = 1}^N c_i \delta_{x_i}$,
 	the averaging  operator  satisfies 
 	$\|A_{r, \mu}\|_{L^1  \to L^1}  \le 
 	E$.
 \end{theorem}

 \begin{proof} Let us  show that  1) $\implies$ 2). Disregarding a set of measure zero if needed, we suppose that 
 	$X = \operatorname{supp} \mu$, so every ball has positive measure.
 	Fix $y \in X$ and $r > 0$. First we consider  the open balls case. Let $0 < s < r$, let 
 	$$g(x) := \frac{\mathbf{1}_{B^o(y,r)}(x)}{\mu B^o(x,r)},   
 	\mbox{ \ and let \ }
 	g_s(x) := \frac{\mathbf{1}_{B^o(y,s)}(x)}{\mu B^o(x,r)}.
 	$$
 	Since balls are open,
 	$g_s\uparrow g$ everywhere as $s\uparrow r$,
 	so we can use the monotone convergence theorem. Thus, it is enough to show that $\lim_{s\to r}
 	\int_{X} g_s d\mu \le E(X,d)$ to conclude that
 	$\|a^o_r\|_{L^\infty(\mu)\to L^\infty (\mu)} \le E(X,d)$. Then  the result follows, 
 	since $\|A_r \|_{L^1(\mu)\to L^1(\mu)} = \|a_r\|_{L^\infty(\mu)\to L^\infty (\mu)} $ by  \cite[Theorem 3.3]{Al2}.	
 	
 	For the next step, we argue as in the proof of 	
 	\cite[Theorem 3.5]{Al2}, which dealt with the case where $X$ is geometrically doubling. We include the details for the readers convenience.
 	Note first  that $b_1 := \inf \{\mu B^o(x,r) : x \in B^o(y,s)\} > 0$. To see why, observe that for every $x\in B^o(y,s)$ and 
 	every $w\in B^o(y, r - s)$, $d(x, w) \le d(x, y) + d (y, w) < s + r - s$,
 	so $ B^o(y, r - s)\subset B^o(x, r)$ and thus $0 < \mu  B^o(y, r - s) \le b_1$.
 	Now take $0 < \varepsilon \ll 1$, and choose $u_1\in B^o(y,s)$ so that
 	$\mu B^o(u_1, r) < (1 + \varepsilon) b_1$; let 
 	$b_2 := \inf \{\mu B^o(x,r) : x \in B^o(y,s) \setminus B^o(u_1, r) \}$, and select
 	$u_2\in B^o(y,s) \setminus B^o(u_1, r) $ so that $\mu B^o(u_2, r) < (1 + \varepsilon) b_2$;
 	repeat, with 
 	$b_{k + 1} := \inf \{\mu B^o(x,r) : x \in B^o(y,s) \setminus \cup_1^k B^o(u_i, r)  \}$, and
 	$u_{k + 1}\in B^o(y,s) \setminus \cup_1^k B^o(u_i, r) $ so that
 	$\mu B^o(u_{k + 1}, r) < (1 + \varepsilon) b_{k + 1}$. Since the 
 	balls $B^o(u_i, r)$ form a Besicovitch family and all contain $y$, 
 	there is an $m\le E(X,d)$ such that
 	$B^o(y,s) \setminus \cup_1^m B^o(u_i, r) = \emptyset$, and then the process stops.

 	Fix $x\in B^o(y,s)$, and let $i$ be the first index such that $x\in B^o(u_i, r)$. Then
 	$$
 	\frac{\mathbf{1}_{B^o(y,s)}(x)}{\mu B^o(x,r)} 
 	\le 
 	(1 + \varepsilon) \frac{\mathbf{1}_{B^o(y,s) \cap B^o(u_i ,r)}(x)}{\mu B^o(u_i ,r)} 
 	\le 
 	(1 + \varepsilon)  \sum_{j=1}^m\frac{\mathbf{1}_{B^o(y,s) \cap B^o(u_j ,r)}(x)}{\mu B^o(u_j ,r)},
 	$$
 	so
 	\begin{equation*}
 	\int_X g_s d\mu   
 	=
 	\int_X  \frac{\mathbf{1}_{B^o(y,s)}(x)}{\mu B^o(x,r)}  \  d\mu(x) 
 	\le
 	\int_X  (1 + \varepsilon)  \sum_{j=1}^m\frac{\mathbf{1}_{B^o(y,s) \cap B^o(u_j ,r)}(x)}{\mu B^o(u_j ,r)}  \  d\mu(x) 
 	\end{equation*}
 	\begin{equation*}
 	\le 
 	(1 + \varepsilon)   \int_X \sum_{j=1}^m\frac{\mathbf{1}_{B^o(u_j ,r)}(x)}{\mu B^o(u_j , r)}  \  d\mu(x) 
 	\le
 	(1 + \varepsilon)  E (X, d),
 	\end{equation*}
 	and now
 	$a^o_r (y) \le E (X,d) 
 	$ follows by letting $\varepsilon \downarrow 0$ and  $s \uparrow r$.
 	
 	The closed balls case is proven using the result for open balls.  Let $0 \le f \in L^1(\mu)$, 
 	$w\in X$, 
 	$R \ge 1$,  and  $\varepsilon > 0$. By monotone convergence, taking $R \uparrow \infty$, it is
 	enough to show that 
 	$$
 	\|\mathbf{1}_{B (w ,R)} \mathbf{1}_{\{ A^{cl}_{r} f \le R\}} A^{cl}_{r} f\|_{L^1} 
 	\le
 	\varepsilon + (1 + \varepsilon) E (X,d) \|f\|_{L^1}.
 	$$
 	For each $x\in B (w ,R)$, choose $r_x > 0$ so that 
 	$\mu B^{o}(x, r + r_x) < (1 + \varepsilon) \mu B^{cl}(x, r)$,
 	and  let
 	$E_n := \{x \in B(w, R) : r_x > 1/n\}$. Then select $N \gg 1$
 	satisfying  
 	$\mu \left(B(w, R)  \setminus E_N\right) < \varepsilon / R$.
 	Now for all $x\in E_N$, we have
 	$$
 	A^{cl}_{r} f (x)
 	\le
 	\frac{1}{ \mu B^{cl}(x, r)} \int_{B^{o}(x, r + 1/N) } f d\mu
 	$$
 	$$
 	\le
 	\frac{1 + \varepsilon}{ \mu B^{o}(x, r + 1/N) } \int_{B^{o}(x, r + 1/N)} f d\mu = (1 + \varepsilon) A^o_{r + 1/N} f(x),
 	$$
 	so
 	$$
 	\|\mathbf{1}_{B (w ,R)} \mathbf{1}_{\{ A^{cl}_{r} f \le R\}} A^{cl}_{r} f\|_{L^1} 
 	\le
 	\|\mathbf{1}_{B (w ,R) \setminus E_N} \mathbf{1}_{\{ A^{cl}_{r} f \le R\}} A^{cl}_{r} f\|_{L^1} + \|\mathbf{1}_{E_N} A^{cl}_{r} f\|_{L^1} 
 	$$
 	$$
 	\le \varepsilon + 
 	(1 + \varepsilon)\|A^o_{r + 1/N} f\|_{L^1} 
 	\le 
 	\varepsilon + (1 + \varepsilon) E (X,d) \|f\|_{L^1}.
 	$$

 	Since 3) is a special  case of 2), the only  implication left is 3) $\implies$ 1); we prove that
 	if  $\mathcal{C}$  is an intersecting   Besicovitch family in $(X,d)$ of equal radius $r$ and
 	cardinality $ > E$, then there exists a discrete measure $\mu_c$ with 
 	finite support, for which 
 	$\|A_{r, \mu_c}\|_{L^1\to L^{1}} >
 	E$. 
 	We may suppose
 	that $\mathcal{C} = \{B(x_1, r), \dots ,B(x_{E + 1}, r)\}$.
 	Let $y \in \cap \mathcal{C}$, and for $0 < c \ll 1$, define 
 	$\mu_c := c \delta_y + \sum_{i= 1}^{L + 1} \delta_{x_i}$. Set
 	$f_c = c^{-1} \mathbf{1}_{\{y\}}$. Then $\|f_c\|_1 = 1$, and
 	for $1 \le i \le E+ 1$, we have $A_{r, \mu_c} f_c (x_i) = 1/(1 + c)$,  while $A_{r, \mu_c} f_c (y) > 0$. Thus
 	$$
 	\|A_{r, \mu_c} f_c\|_1 >
 	\frac{E + 1}{ 1 + c},
 	$$
 	and the result follows by taking 
 	$c$ small enough.
 \end{proof}
 
 Since $\|A_{r,\mu} \|_{L^\infty(\mu)\to L^\infty(\mu)}  =1$, by  interpolation
 or  by Jensen's inequality (cf. \cite[Theorem 2.10]{Al3})
 for all  $1 < p < \infty$, we have $\|A_{r,\mu}  \|_{L^p(\mu)\to L^p(\mu)}  \le  E(X,d)^{1/p}$. 
 
 \begin{remark} In addition to  having $\sup_{r, \mu}\|A_{r,\mu}  \|_{L^1(\mu)\to L^1(\mu)} =  E(X,d)$, using the same measures and functions it is easy to see that equality also holds for the weak type $(1,1)$ bounds, that is, 
 	$\sup_{r, \mu}\|A_{r,\mu}  \|_{L^1(\mu)\to L^{1, \infty}(\mu)} =  E(X,d)$. In fact, since the function $f_c$ is a scalar multiple of an indicator function, this equality holds in the restricted weak type (1,1) case.
 \end{remark}
 
 The preceding theorem entails that the uniform weak type $(1,1)$ of the centered
 maximal operator is stronger than the uniform strong type $(1,1)$ of the averaging
 operators.
 
 \begin{corollary}  \label{weakstrong}
 	Given any metric space   $(X, d)$, we have 
 	$$
 	\sup_{r, \mu}\|A_{r, \mu }\|_{L^1\to L^{1}} \le \sup_{\mu} \|M_{\mu}\|_{L^1  \to L^{1,\infty}},
 	$$
 	where the supremum on the left hand side is taken over all $r > 0$ and all $\tau$-additive, locally finite Borel measures $\mu$ on $X$, and the supremum on the right, over all such $\mu$.
 \end{corollary} 
 
 \begin{corollary}  \label{L1conv} Let $1 \le p < \infty$.
 	If  $(X, d)$ has the equal radius Besicovitch intersection property, and $\mu$ is a
 	$\tau$-additive Borel measure on $X$, then for  every $f\in L^p(\mu)$,  
 	we have $\lim_{r\to 0}  A_{r} f  =  f$ in $L^p$.
 \end{corollary} 
 
 This corollary follows in a standard fashion from the uniform boundedness of
 the averaging operators (cf. \cite{Al2} for more details).
 
 \vskip .2 cm
 
 Analogously to the case of the centered maximal operator, given any $p \in (1, \infty)$, the uniform weak
 type $(p,p)$  implies the  equal radius Besicovitch intersection property, and consequently, one can extrapolate from uniform
 weak type $(p,p)$ to uniform strong type $(1,1)$.

 \begin{theorem}\label{extrapolation} Let $(X, d)$ be a  
 	metric  space. Each of the following statements implies the next:
 	
 	1)	There exist a $p$ with $1 < p < \infty$ and an integer $N\ge 1$, such that
 	for every discrete, finite  Borel measure $\mu$ with finite support
 	in $X$, and every $r > 0$, the averaging operators
 	$A_{r, \mu}$ satisfy 
 	$\|A_{r, \mu}\|_{L^p\to L^{p,\infty}} \le N$.
 	
 	2) The space $(X, d)$ has the  equal radius Besicovitch intersection property with constant 
 	$\lfloor p^p (p - 1)^{(1-p)} N^p \rfloor$.
 	
 	3)  For every $\tau$-additive, locally finite  Borel measure $\mu$ on $X$ and every $r > 0$,  the averaging operators
 	$A_{r, \mu}$ satisfy 
 	$\|A_{r, \mu}\|_{L^1\to L^{1}} \le 
 	\lfloor p^p (p - 1)^{(1-p)} N^p \rfloor$.	
 \end{theorem}

 \begin{proof} The implication   2) $\implies$ 3) is part of Theorem \ref{discreteequivav}.  Regarding
 	1) $\implies$ 2),  we show that
 	if  $\mathcal{C}$  is an intersecting  Besicovitch family  in $(X,d)$, of
 	cardinality  strictly larger than $\lfloor p^p (p - 1)^{(1-p)} N^p\rfloor$ and equal radius $r$, then there exists a finite
 	sum of weighted Dirac deltas $\mu$, for which 
 	$\|A_{r, \mu}\|_{L^p\to L^{p,\infty}} >
 	N$. 
 	
 	Let $q = p/(p - 1)$ be the dual exponent of $p$, and let $J:= \lfloor p^p (p - 1)^{(1-p)} N^p \rfloor + 1$.	
 	We may suppose
 	that $\mathcal{C} = \{B(x_1, r), \dots, B(x_{J}, r)\}$.
 	Let $y \in \cap \mathcal{C}$, and set, for $ c > 0$, 
 	$\mu_c :=  c \delta_y + \sum_{i= 1}^{J} \delta_{x_i}$.
 	Recall that for every $\alpha > 0$, by definition of the weak $(p,p)$ constant $\|A_{r, \mu_c}\|_{L^p\to L^{p,\infty}}$,
 	\begin{equation}\label{weaktypepM}
 	\mu_c (\{A_{r, \mu_c} f \ge \alpha\}) \le \left(\frac{\|A_{r, \mu_c}\|_{L^p\to L^{p,\infty}} \|f\|_{L^p}}{\alpha}\right)^p.
 	\end{equation}
 	Set
 	$f = \mathbf{1}_{\{y\}}$; then $\|f\|_{L^p(\mu_c)} = c^{1/p}$. For $1 \le i \le J$, we have $A_{r, \mu_c} f (x_i) = c/(1 + c)$. 
 	Thus,
 	$\mu_c \{A_{r, \mu_c} f \ge  c/(1 + c)\} = J $,  so 
 	with $\alpha = c/(1 + c)$, we have 
 	\begin{equation}\label{weaktype}
 	J^{1/p} 
 	\le
 	\frac{\|A_{r, \mu_c}\|_{L^p\to L^{p,\infty}} (1 + c)}{c^{1/q}}.
 	\end{equation}
 	As before, maximizing $g(c) = c^{1/q}/(1 + c)$ we get $c = p - 1$
 	and $g(p - 1) = (p - 1)^{(p - 1)/p} p^{-1}$,  so
 	\begin{equation}\label{weaktype1}
 	N 
 	=
 	\frac{(p -1)^{(p - 1)/p}  \left(p^p (p - 1)^{(1-p)} N^p\right)^{1/p} }{p} 
 	< \frac{(p -1)^{(p - 1)/p}  \ J ^{1/p} }{p} 
 	\le
 	\|A_{r, \mu_c}\|_{L^p\to  L^{p,\infty}}.
 	\end{equation}
 \end{proof}
 
 In what follows  we take balls to be closed. Unlike the case of the Heisenberg groups (where we have $E(X,d) < \infty$ and $L(X,d) = \infty$) in Banach spaces the equality $E(X,d)  = L(X,d)$ always holds.
 
 \begin{theorem} \label{kissme}  If $(X, \|\cdot\|)$ is a
 	Banach space, then  $E(X, \| \cdot\|) = L(X, \| \cdot\|)$. 
 \end{theorem}
 
 \begin{proof}  It suffices to show that 
 	$E(X, \| \cdot\|) \ge L(X, \| \cdot\|)$. Both $E (X, \| \cdot\|)$ and $L (X, \| \cdot\|)$ are defined as suprema, so it is enough to prove that given any finite, intersecting Besicovitch family 
 	$\mathcal{C} := \{B^{cl} (x_1, r_1),  \dots,  B^{cl} (x_n, r_n) \}$, we can produce an equal radius  intersecting 
 	Besicovitch family  of the same cardinality. Choose $y \in \cap \mathcal{C}$; 
 	by a translation and a dilation, we may assume that
 	$y = 0$ and $r = 1$. By repeating the argument from the proof of Theorem  \ref{kiss}, we see that  
 	$\mathcal{C}^\prime := \{B^{cl} (x_1/\|x_1\|, 1),  \dots,  B^{cl} (x_n/\|x_n\|, 1) \}$ is the desired equal radius Besicovitch family. 
 \end{proof}
 
 As is the case with the maximal operator,  in $\mathbb{R}^d$  it is possible to
 construct a measure $\mu$ for which 
 the supremum  is attained, with $r = 1$. We omit the proof.
 
 \begin{theorem} \label{kissmenot}  Let $\|\cdot\|$ be any norm on 
 	$\mathbb{R}^d$. Then there exists a discrete measure $\mu$ 
 	such that $\|A^{cl}_{1,\mu}  \|_{L^1(\mu)\to L^1(\mu)} =  E(\mathbb{R}^d, \| \cdot\|)$.
 \end{theorem}
 
 The equality  $E(X, \| \cdot\|) = L(X, \| \cdot\|)$
 allows one to transfer uniform bounds known for the centered maximal operator to uniform bounds for the averaging operators. 
 
  \begin{remark} It is obvious that $E(\mathbb{R}, \| \cdot\|) = 2$, so  $\sup_{r, \mu}\|A_{r,\mu}  \|_{L^1(\mu)\to L^1(\mu)} = 2$ for every locally finite Borel measure $\mu$ on  $\mathbb{R}$. The special case 
  	$\|A_{r,\mu}  \|_{L^1(\mu)\to L^1(\mu)} \le 2$ for $\mu$ the standard exponential
  	distribution, given by $d \mu (t) = \mathbf{1}_{(0,\infty)} (t) \ e^{-t} dt$, had already been proven  in  \cite[Theorem 4.2]{Al3}.
 \end{remark}

 In higher dimensions, from Corollaries \ref{2bounds}, \ref{infinitybounds}, and \ref{bounds},  we obtain the following
 
 \begin{corollary}  Given any norm $
 	\|\cdot\|
 	$ on the plane,  if the unit ball is a parallelogram then   $\sup_{r, \mu}\|A_{r,\mu}  \|_{L^1(\mu)\to L^1(\mu)} = 4
 	$, while 
 	$\sup_{r, \mu}\|A_{r,\mu}  \|_{L^1(\mu)\to L^1(\mu)} =  5$ in every other case.

 	With balls defined using the $\ell_\infty$ norm,  the sharp uniform bound for $\sup_{r, \mu}\|A_{r,\mu}  \|_{L^1(\mu)\to L^1(\mu)}$ on $(\mathbb{R}^d , \|\cdot\|_\infty)$ 
 	is $2^d$. 
 	
 	For the euclidean norm we have  
 	$\sup_{r, \mu}\|A_{r,\mu}  \|_{L^1(\mu)\to L^1(\mu)} =   12$ in dimension 3. Asymptotically, in dimension $d$ the following bounds hold:
 	\begin{equation}\label{asym}
 	(1 + o(1)) \sqrt{\frac{3 \pi}{8}}  \log {\frac{3}{2 \sqrt 2}}
 	\ 	d^{3/2} \  \left(\frac{2}{\sqrt{3}}\right)^d
 	\le
 	\sup_{r, \mu}\|A_{r,\mu}  \|_{L^1(\mu)\to L^1(\mu)} \le
 	2^{0.401 (1 + o(1)) d}.
 	\end{equation}
 \end{corollary}

 \begin{remark} For $\mathbb{R}^d$ with the euclidean norm and
 	the standard gaussian measure $\gamma$, it was shown in \cite[Theorem 4.3]{Al3}
 	that 
 	$\sup_{r>0}\|A_r\|_{L^1(\gamma)\to L^1(\gamma)} \le (2 + \varepsilon)^d$, whenever  $\varepsilon > 0$ and $d$ is large enough.  
 	Note that the  upper bounds from the preceding result
 	(valid for all measures) represent a substantial improvement.
 \end{remark}

\end{document}